\documentclass[11 pt]{amsart}
\usepackage{amssymb}
\usepackage{graphicx}
\usepackage{color}
\usepackage{amsmath}
\usepackage{graphicx}
\usepackage{mathrsfs}
\usepackage{amsfonts}
\usepackage{amscd}
\usepackage{epic}
\newtheorem{theorem}{Theorem}[section]

\newtheorem{definition}[theorem]{Definition}

\newtheorem{lemma}[theorem]{Lemma}
\newtheorem{remark}[theorem]{Remark}
\newtheorem{proposition}[theorem]{Proposition}

\numberwithin{equation}{section}

\renewcommand{\H}{{\mathcal H}}

\def\R{\mathbb R}

\def\R{\mathbb R}

\def\N{\mathbb N}

\def\de{\delta}

\def\la{\lambda}

\def\p{\mathfrak{p}}

\def\la{\lambda}

\def\N{\mathbb{N}}

\def\R{\mathbb{R}}

\def\nn{\nonumber}
\def\noop#1{\Vert #1\Vert_{\rm op}}

\def\R{{\mathbb R}}

\def\N{{\mathbb N}}

\def\HH{{\mathbb H}}

\def\B{{\mathcal B}}
\def\D{{\mathcal D}}
\def\F{{\mathcal F}}

\def\H{{\mathcal H}}

\def\hb#1{\hbox{#1}}

\def\no#1{\Vert #1\Vert }

\def\ker#1{\hb{ker}(#1)}

\def\res#1{_{\vert #1}}

\def\hb#1{\hbox{#1}}

\def\ker#1{\hb{ker}(#1)}

\def\L1#1{L^1(#1)}

\def\L#1#2{L^{#1}(#2)}

\def\lef({\left(}
\def\rig){\right)}

\begin{document}
\title{ The  $C^*$-algebra of the Cartan motion groups }
\author{Regeiba Hedi \and Rahali Aymen}
\address{ Universit\'{e} de Sfax, Facult\'{e} des Sciences Sfax,  BP
1171, 3038 Sfax, Tunisia.} \email{aymenrahali@yahoo.fr}
\address{ Universit\'{e} de Sfax, Facult\'{e} des Sciences Sfax,  BP
1171, 3038 Sfax, Tunisia\\
Université de Gabés
Facult\'e des Sciences de Gab\'es
Cit\'e Erriadh 6072 Zrig Gab\'es Tunisie.
} \email{rejaibahedi@gmail.com}
 \maketitle{}
\begin{abstract}
Let $G_0=K\ltimes\p$ be the Cartan motion groups. Under some assumption on $G_0,$ we describe the $C^*$-algebra
$C^*(G_0)$ of $G_0$ in terms of operator fields.

\end{abstract}
\section{\bf{Introduction}}\label{sec:1}
Let $G$ be a locally compact group. We denote by $\widehat{G}$ the unitary dual of $G.$ It well-known
that $\widehat{G}$ equipped with the Fell topology (see \cite{Mac1,Mac2}). The first representation-theoretic question concerning
the group $G$ is the full parametrezation and topological identification of the dual $\widehat{G}.$ The $C^*$-algebra
$C^*(G)$ is the completion of the convolution algebra $L^1(G)$ equipped with the $C^*$-norm $\|.\|_{C^*(G)}$,
given by $$\|f\|_{C^*(G)}:=\sup_{\pi\in\widehat{G}}\|\pi(f)\|_{op}.$$ We denote by  $\widehat{C^*(G)}$ the unitary dual of the
$C^*$-algebra of $G.$ Then we have the following bijection $$\widehat{C^*(G)}\simeq\widehat{G}.$$ Furthermore,
the $C^*$-algebra $C^*(G)$ of $G$ can be identified with a subalgebra of the large $C^*$-algebra $\ell^\infty(\widehat{G})$ of
bounded operator fields given by\\
{\footnotesize
 $$ \ell^\infty(\widehat{G}) :=
  \Big\{F:\widehat{G}\longrightarrow\bigcup_{\pi\in\widehat{G}}\mathcal{B}(\mathcal{H}_\pi),\,\pi\longmapsto
  F(\pi)\in\mathcal{B}(\mathcal{H}_\pi);\,\|F\|_\infty:=\sup_{\pi\in\widehat{G}}\|F(\pi)\|_{op}<\infty\Big\}$$}
under the Fourier transform $\mathcal{F}$ defined on $C^*(G)$ as follows:
$$\mathcal{F}(f)(\pi)=\pi(f),\,\, \pi\in\widehat{G},\, f\in C^*(G).$$ Using the fact that $\mathcal{F}$
is an injective homomorphism of $C^*(G)$ into $\ell^\infty(\widehat{G}),$ then the $C^*$-algebra $C^*(G)$ is isomorphic to
a subalgebra $\mathcal{D}:=\mathcal{F}(C^*(G))$ of elements in $\ell^\infty(\widehat{G})$ verifying some conditions.
The elements of $\mathcal{D}$ must naturally fulfil is that of continuity. Then the parametrization and the descreption of
the topology of $\widehat{G}$ are required to describe the $C^*$-algebra $C^*(G)$ of $G.$\\
In this context, we have some works in the literature , for example,
J. Ludwig and L. Turowska have described in \cite{Lud-Tur} the $C^*$-algebra of the Heisenberg group and of
the thread-like Lie groups in terms of an algebra of operator fields defined over their dual spaces.
The descreption of the $C^*$-algebra of the Euclidean motion group
$M_n:=SO(n)\ltimes\mathbb{R}^n,\,n\in\mathbb{N}^*$ was established in \cite{Lud-Ell-Abd}.
In the present work, we give a similar precise descreption of the $C^*$-algebra of the Cartan motion groups.
Our result is a generalization of analogous results in the case of the Euclidean motion group (see \cite{Lud-Ell-Abd}).\\
The paper is organized as follows.
In section \ref{sec:2}, we introduce the groups $G_0=K\ltimes \p$,
the semi-direct product  of the maximal compact connected subgroup $K $ of same connected semisimple Lie group with finite center $G$,
acting by adjoint action on $\p$ (where $\p$ determined by the Cartan decomposition of Lie algebra of the group $G$).
We recall the topology of the spectrum of the groups $G_0 $ and we determine the convergence in $\widehat{G_0}$.
 In the last section,  we determine the Fourier transform of their group $C^*$-algebras and
we describe the elements of the image of the Fourier transform of $C^*(G_0)$ inside the big algebra $\ell^\infty(\widehat{G_0}).$
  This is the main result of the paper

\section{\bf{The Cartan motion groups $G_0$.}}\label{sec:2}
Let $G$ be a connected semisimple Lie group with finite center and $K$ a maximal compact connected
subgroup of $G.$ Let $\mathfrak{g}=\mathfrak{k}\oplus\mathfrak{p}$ be the corresponding Cartan decomposition of
the Lie algebra $\mathfrak{g}$ of $G$ with $\mathfrak{k}:=Lie(K).$ Then one can form the semidirect product $G_0:=K\ltimes\mathfrak{p}$
with respect to the adjoint action of $K$ on $\mathfrak{p}.$ The group $G_0$ is called the Cartan motion group associated to the
Riemannian symmetric pair $(G,K).$ The multiplication in this group is given by
$$(k_{1},X_{1})\cdot(k_{2},X_{2})=(k_{1}k_{2},X_{1}+Ad(k_{1})X_{2}),\,\, (k_1,k_2)\in K, (X_1,X_2)\in\mathfrak{p}.$$
The group $M_{n}=SO(n)\ltimes\mathbb{R}^{n}$ is an example of Cartan motion groups. More precisely, $M_{n}$ is
the Cartan motion group associated to the compact Riemannian symmetric pair $(SO(n+1),SO(n))$.\vspace{0,2cm}\\
Let now $\mathfrak{a}$ be a maximal abelian subspace of $\mathfrak{p}$. The dimension of the real vector space $\mathfrak{a}$
is called the rank of the Riemannian symmetric pair $(G,K)$. An important fact worth mentioning here is that every adjoint orbit
of $K$ in $\mathfrak{p}$ intersects $\mathfrak{a}$ (see \cite{Hel}  p.\:247 ). Let $N_{K}(\mathfrak{a})$ and $Z_{K}(\mathfrak{a})$
denote respectively the normalizer and centralizer of $\mathfrak{a}$ in $K.$ The quotient group $$W:=N_{K}(\mathfrak{a})/Z_{K}(\mathfrak{a})$$
is called the Weyl group of the pair $(G,K)$. We shall denote the action of $W$ on $\mathfrak{a}$ by $H\longmapsto s.H$ for $H\in\mathfrak{a}$
and $s\in W$.\vspace{0,1cm}
Let us take the subspaces $\tilde{\mathfrak{a}}:=i\mathfrak{a}$, $\tilde{\mathfrak{p}}:=i\mathfrak{p}$ and
$\tilde{\mathfrak{g}}:=\mathfrak{k}\oplus\tilde{\mathfrak{p}}$ of the complexification $\mathfrak{g}^{\mathbb{C}}$ of $\mathfrak{g}$.
An element $H\in\mathfrak{a}$ is called regular if $\alpha(H)\neq0$ for all $\alpha\in\Sigma$ where $\Sigma$ is the set of all
restricted roots associated to the pair $(\mathfrak{g},\mathfrak{a}).$
A connected component of the set of regular elements in $\mathfrak{a}$ is called a Weyl chamber of the pair $(G,K)$.
Endow the dual space $\tilde{\mathfrak{a}}^{*}$ with a lexicographic ordering and denote by $\Sigma^{+}$ the set of positive restricted roots.
As an example of Weyl chambers, let us set $C^{+}(\mathfrak{a}):=i\,\widetilde{C}^{+}(\tilde{\mathfrak{a}})$ with
$$\widetilde{C}^{+}(\tilde{\mathfrak{a}})=\lbrace\widetilde{H}\in\tilde{\mathfrak{a}};\:\alpha(\widetilde{H})>0,\:
\forall\alpha\in\Sigma^{+}\rbrace.$$
It is well-known that every $s\in W$ permutes the Weyl chambers and that $W$ acts simply transitively on the set of Weyl chambers.
Furthermore, we have the following important result (see \cite{Hel}, p.\:322):
\begin{proposition}\label{prop1} Let $C\subset\mathfrak{a}$ be a Weyl chamber. Each orbit of $W$ in $\mathfrak{a}$ intersects the
closure $\overline{C}$ in exactly one point.
\end{proposition}
We shall briefly review the description of the unitary dual of $G_{0}$ via Mackey's little group theory.
Let $\varphi$ be a non-zero linear form on $\mathfrak{p}$. We denote by $\chi_{\varphi}$ the unitary character of the vector Lie group
$\mathfrak{p}$ given by $\chi_{\varphi}=e^{i\varphi}$. Let $K_{\varphi}$ be the stabilizer of $\varphi$ under the coadjoint action of
$K$ on $\mathfrak{p}^{*}$, and let $\rho$ be an irreducible unitary representation of $K_{\varphi}$ on some Hilbert space $\mathcal{H}_{\rho}$.
The map
$$\rho\otimes\chi_{\varphi}: (k,X)\longmapsto e^{i\varphi(X)}\rho(k)$$
is a representation of the semidirect product $K_{\varphi}\ltimes\mathfrak{p}$, which we may induce up so as to obtain a unitary
representation of $G_{0}$. Let $L_{\rho}^{2}(K,\mathcal{H}_{\rho})$ be the subspace of $L^{2}(K,\mathcal{H}_{\rho})$ consisting of
the maps $f$ which satisfy the covariance condition
$$f(kk_{0})=\rho(k_{0}^{-1})f(k)$$
for $k_{0}\in K_{\varphi}$ and $k\in K$. The induced representation
$$\pi_{(\rho,\varphi)}:=Ind_{K_{\varphi}\ltimes\mathfrak{p}}^{\,G_{0}}(\rho\otimes\chi_{\varphi})$$
is realized on $\mathcal{H}_{\rho,\varphi}:=L_{\rho}^{2}(K,\mathcal{H}_{\rho})$ by
$$\pi_{(\rho,\varphi)}(k_{0},X)f(k)=e^{i\varphi(Ad(k^{-1})X)}f(k_{0}^{-1}k),$$
where $(k_{0},X)\in G_{0}$, $f\in L_{\rho}^{2}(K,\mathcal{H}_{\rho})$ and $k\in K$.
Mackey's theory tells us that the representation $\pi_{(\rho,\varphi)}$ is irreducible and that every infinite dimensional
irreducible unitary representation of $G_{0}$ is equivalent to some $\pi_{(\rho,\varphi)}$.
Furthermore, two representations $\pi_{(\rho,\varphi)}$ and $\pi_{(\rho^{\,'},\varphi^{\,'})}$
are equivalent if and only if $\varphi$ and $\varphi^{\,'}$ lie in the same coadjoint orbit of $K$ and
the representations $\rho$ and $\rho^{\,'}$ are equivalent under the identification of the conjugate
subgroups $K_{\varphi}$ and $K_{\varphi^{\,'}}$. In this way, we obtain all irreducible representations
of $G_{0}$ which are not trivial on the normal subgroup $\mathfrak{p}$. On the other hand, every irreducible unitary
representation $\tau$ of $K$ extends trivially to an irreducible representation, also denoted by $\tau$, of
$G_{0}$ by $\tau(k,X):=\tau(k)$ for $k\in K$ and $X\in\mathfrak{p}$.\vspace{0,1cm}

Next, we shall provide a more precise description of the so-called ``generic irreducible unitary representations" of $G_{0}$.
Denote again by $\langle\:,\:\rangle$ the restriction to $\mathfrak{p}\times\mathfrak{p}$ of the $Ad(K)$-invariant scalar product
$\langle\:,\:\rangle$ on $\mathfrak{g}_{0}$. Let $\mathfrak{a}$ be a maximal abelian subspace of $\mathfrak{p}$, and let $M$ be the centralizer
of $A=exp_{{}_{G}}(\mathfrak{a})$ in $K$. In general, the compact Lie group $M$ is not connected, and one can prove that
$M=M_{e}\cdot(M\cap A)$ with $M_{e}$ being the identity component of $M$. A proof of the following well-known result can be found in \cite{Ben-Rah}.
\begin{lemma}\label{lem1}
Let $C\subset\mathfrak{a}$ be a Weyl chamber. Every adjoint orbit of $K$ in $\mathfrak{p}$ intersects the closure
$\overline{C}$ in exactly one point.
\end{lemma}

We conclude that every infinite dimensional unitary representation of $G_{0}$ has the form $\pi_{(\rho,\varphi_{_{H}})}$,
where $H$ is a non-zero vector in $\overline{C^{+}(\mathfrak{a})}$ and $\varphi_{{}_{H}}$ is the linear form on $\mathfrak{p}$
given by $\varphi_{{}_{H}}(X)=\langle H,X\rangle$. Observe that the isotropy group $K_{\varphi_{{}_{H}}}$ coincides with the centralizer
$Z_{K}(H)$. Let us fix a regular element $H$ in $\mathfrak{a}$. The subgroups $K_{\varphi_{{}_{H}}}$ and $M$ of $K$ are identical.
If $\rho$ is an irreducible representation of $M$, then the representation $\pi_{(\rho,\varphi_{{}_{H}})}$ corresponding
to the pair $(\rho,\varphi_{{}_{H}})$ is said to be generic. We denote by $\Gamma_0$ the set of all equivalence classes of generic irreducible
unitary representations of $G_{0}$. Notice that $\Gamma_0$ has full Plancherel measure in the unitary dual $\widehat{G_{0}}$.
Applying Mackey's analysis and the result of Lemma \ref{lem2}, we obtain the bijection
$$\Gamma_0\simeq \widehat{M}\times C^{+}(\mathfrak{a}).$$
 Note that when the Riemannian symmetric pair $(G,K)$ has rank one, we can find a unit vector $H_{0}\in\mathfrak{a}$ such that $C^{+}(\mathfrak{a})=\mathbb{R}_{+}^{*}H_{0}$. In this case we have the bijections
$$\Gamma_0\simeq \widehat{M}\times\mathbb{R}_{+}^{*}\:\:\:\text{and}\:\:\:\widehat{G_{0}}\simeq \Gamma_0\sqcup \Gamma_2$$ where $\Gamma_2:=\widehat{K}.$
Now, we define the subset $\Gamma_1\subset\widehat{G_0}$ as follows $$\Gamma_1:=\Big\{\pi_{(\mu,H)}=Ind_{K_{\varphi_H}\ltimes\mathfrak{p}}^{K\ltimes\mathfrak{p}}(\rho\otimes\chi_{\varphi_H}),\,\,\rho\in\widehat{K_{\varphi_H}},\,H\in \partial(C^+(\mathfrak{a}))\backslash\{0\}\Big\}$$ where $\partial(C^+(\mathfrak{a}))$ is the boundary of $C^+(\mathfrak{a}).$ According to Mackey's theory, we obtain the following parametrization as sets $$\widehat{G_0}\simeq \Gamma_0\sqcup\Gamma_1\sqcup\Gamma_2.$$
We denote by $\mathfrak{z}$ the orthogonal complement of $\mathfrak{a}$ in $\mathfrak{p}$ ($\mathfrak{p}=\mathfrak{a}\oplus\mathfrak{z}$). Let $\Lambda:\mathfrak{a}\longrightarrow\mathbb{C}$ be a real linear function. Also we denote by $\Lambda$ the extension of $\Lambda$ to $\mathfrak{p}$ so that $\mathfrak{z}\subseteq Ker(\Lambda),$ and let $\rho\in\widehat{K_\Lambda}.$ We denote by $\pi:=\pi_{(\rho,\Lambda)}$ the representation of $G_0$ induced from \begin{eqnarray*}
  K_\Lambda\ltimes\mathfrak{p}&\longrightarrow& \mathcal{L}(E_\rho) \\
  (k,X) &\longmapsto& e^{\sqrt{-1}\Lambda(X)}\rho(k).
 \end{eqnarray*}
 Now, we describe the Fell topology on $\widehat{G_0}.$ For $\beta_1, \beta_2\in\mathfrak{a}^*$ (the dual vector space of $\mathfrak{a}$), define $$|\beta_1+\sqrt{-1}\beta_2|^2=B(\beta_1,\beta_1)+B(\beta_2,\beta_2)$$ where $B$ is the Cartan-Killing form on $\mathfrak{g}_0.$
Let $\mathcal{F}_c$ be the set of all pairs $(\rho,\Lambda)$ where $\rho\in\widehat{K}_\Lambda.$
 We take $(\rho,\Lambda)\in\mathcal{F}_c,$ if $\varepsilon>0$ is sufficiently small then $|\Lambda-\Lambda^{'}|<\varepsilon$ implies $K_{\Lambda^{'}}\subseteq K_\Lambda.$
 So the subset $$\mathcal{U}:=\Big\{(\rho^{'},\Lambda^{'})\in\mathcal{F}_c:\,|\Lambda-\Lambda^{'}|<\varepsilon\,\,and\,\,[\rho\big|_{K_{\Lambda^{'}}}:\rho^{'}]>0\Big\}$$ ($[\rho\big|_{K_{\Lambda^{'}}}:\rho^{'}]$ is the multiplicity of $\rho^{'}$ in $\rho\big|_{K_\Lambda}$)
defines a basis for the neighborhoods of $(\rho,\Lambda)$ in the topology we give $\mathcal{F}_c$ (see \cite{Rad}).
Note that $W$ acts on $\mathcal{F}_c$ by $$w.(\rho,\Lambda)=(w.\rho,w.\Lambda).$$ Let $\mathcal{F}_c/W$
be the quotient space by this action of $W,$ equipped with the quotient topology. Now,
let $$\mathcal{F}:=\Big\{(\rho,\Lambda)\in\mathcal{F}_c:\,\Lambda=\sqrt{-1}\beta\,\,where\,\,\beta\,\,is\,\,real\,\,valued\Big\}.$$
According to \cite{Rad}, then we have the useful Lemma.
\begin{lemma}\label{lem2}
The unitary dual $\widehat{G_0}$ of $G_0$ is homeomorphic to $\mathcal{F}/W.$
\end{lemma} In the remainder of this paper, we shall assume that the stabilizer $K_{\varphi}$ is connected for each $\varphi\in\mathfrak{p}^*.$ Let $\rho_{\mu}$ be an irreducible representation of $K_H$ with highest weight $\mu$. For simplicity,
we shall write $\pi_{(\mu,H)}$ and $\mathcal{H}_{(\mu,H)}$ instead of $\pi_{(\rho_{\mu},\varphi_{{}_{H}})}$ and
$\mathcal{H}_{\rho_\mu,\varphi_{H}}$ respectively.
We have:
\begin{proposition}\label{prop2} Let $(\pi_{(\mu^n,H_n)})_n$ be a sequence in $\Gamma_0.$ Then we have:
\begin{enumerate}
 \item The sequence $(\pi_{(\mu^n,H_n)})_n$ converges to $\pi_{(\mu,H)}$ in $\Gamma_0$ if and only if $(H_n)_n$ converges to $H$ and $\mu^n=\mu$ for $n$ large enough.
 \item The sequence $(\pi_{(\mu^n,H_n)})_n$ converges to $\pi_{(\mu,H)}$ in $\Gamma_1$ if and only if $(H_n)_n$ converges to $H$ and $[\rho_{\mu}\big|_M:\rho_{\mu^n}]>0$ for $n$ large enough.
 \item The sequence $(\pi_{(\mu^n,H_n)})_n$ converges to $\tau_\lambda$ in $\Gamma_2$ if and only if $(H_n)_n$ converges to $0$ and $[\tau_\lambda\big|_{M}:\rho_{\mu^n}]>0$ for $n$ large enough.
\end{enumerate}

\end{proposition}
\begin{proof}
By Lemma \ref{lem2}, we show that the map \begin{eqnarray*}
 \mathcal{F}/W &\longrightarrow& \widehat{G}_0 \\
  (\rho, \Lambda) &\longmapsto& \pi_{(\rho,\Lambda)}
  \end{eqnarray*}
is a homeomorphism (see  \cite{Rad}). Let $H$ be a non-zero vector in $\overline{C^+(\mathfrak{a})}$ and $\Lambda_H$ is the linear form on $\mathfrak{a}$ given by $$\Lambda_H(X)=\sqrt{-1}\langle H,X\rangle\,\,\,\forall X\in\mathfrak{a}.$$ For simplicity, we shall write $K_H$ instead of $K_{\Lambda_H}.$ Now, we assume that $(H_n)_{_n\in I}$ converges to $H$ and $[\rho_\mu\big|_{K_{H_n}}:\rho_{\mu^n}]>0$ for $n$ large enough. This is equivalent that the net $(\rho_{\mu^n},\Lambda_{H_n})_{n\in I}$ converges to $(\rho_\mu,\Lambda_H)$ in $\mathcal{F}/W.$ In view of the continuity of the map $\mathcal{F}/W \ni(\rho, \Lambda) \longmapsto \pi_{(\rho,\Lambda)}\in\widehat{G_0},$ we easily see that $(\pi_{(\mu^n,H_n)})_n$ converges to $\pi_{(\mu,H)}$ in $\widehat{G_0}.$\\
Conversely, assume that the net $(\pi_{(\mu^n,H_n)})_{n\in I}$ converges to $\pi_{(\mu,H)}$ in $\widehat{G_0}.$ Then the net $((\rho_{\mu^n},\Lambda_{H_n}))_{n\in I}$ converges to $(\rho_\mu,\Lambda_H)$ in $\mathcal{F}/W.$ Recall that the set
$$\mathcal{V}:=\Big\{(\rho^{'},\Lambda^{'})\in\mathcal{F}_c:\,|\Lambda_H-\Lambda^{'}|<\varepsilon\,\,and\,\,[\rho_\mu\big|_{K_{\Lambda^{'}}}:\rho^{'}]>0\Big\}$$ defines a basis for the neighborhoods of $(\rho_\mu,\Lambda_H).$ Hence for each $\varepsilon>0,$ there exists $n_0\in I$ such that $$\forall n\geq n_0,\,\,\, (\rho_{\mu^n},\Lambda_{H_n})\in \mathcal{V}.$$ i.e.; $\forall n\geq n_0,$ we have

 \begin{eqnarray*}
 |\Lambda_{H_n}-\Lambda_H| &<& \varepsilon
 \end{eqnarray*} and \begin{eqnarray*}
                      [\rho_\mu\big|_{K_{H_n}}:\rho_{\mu^n}] &>&0 .
                    \end{eqnarray*}
                    Then we obtain the following

 \begin{eqnarray*}
 H_n &\longrightarrow& H
 \end{eqnarray*} and \begin{eqnarray*}
                      [\rho_\mu\big|_{K_{H_n}}:\rho_{\mu^n}] &>&0 .
                    \end{eqnarray*} for $n$ large enough.
Notice that for $H, H_n\in C^+(\mathfrak{a}),$ $K_H=K_{H_n}=M$ and we get $$[\rho_\mu\big|_{K_{H_n}}:\rho_{\mu^n}] >0 \Longleftrightarrow  [\rho_\mu\big|_M:\rho_{\mu^n}] >0$$ for $n$ large enough, which is equivalent to $\mu^n=\mu$ for $n$ large enough. This completes the proof of the Proposition.
\end{proof}
\begin{proposition}\label{prop3} Let $(\pi_{(\mu^n,H_n)})_n$ be a sequence in $\Gamma_1.$ Then we have:
\begin{enumerate}
 \item The sequence $(\pi_{(\mu^n,H_n)})_n$ converges to $\pi_{(\mu,H)}$ in $\Gamma_1$ if and only if $(H_n)_n$ converges to $H$ and $[\rho_{\mu}\big|_{K_{H_n}}:\rho_{\mu^n}]>0$ for $n$ large enough.
 \item The sequence $(\pi_{(\mu^n,H_n)})_n$ converges to $\tau_\lambda$ in $\Gamma_2$ if and only if $(H_n)_n$ converges to $0$ and $[\tau_\lambda\big|_{K_{H_n}}:\rho_{\mu^n}]>0$ for $n$ large enough.
\end{enumerate}
\end{proposition}
\begin{proof} Applying the same arguments as in the proof of Proposition 2.4.

\end{proof}
Of course $\widehat{K}$ has the discrete topology.\vspace{0,2cm}\\
Let $C^*(G_0)$ denote the full $C^*$-algebra of $G_0.$
We denote by $C_0(\widehat{K})$ the Banach algebra of all operator fields
$$F:\widehat{K}\longrightarrow\bigcup_{\pi\in\widehat{K}}\mathcal{B}(\mathcal{H}_\pi),$$ such that
$F(\pi)\in\mathcal{B}(\mathcal{H}_\pi)$ for each irreducible unitary representation $\pi$ of $K$ and such that
$\lim_{\pi\longrightarrow\infty}\|F(\pi)\|_{op}=0.$ This algebra is equipped with the norm
$\|F\|_\infty=\sup_{\pi\in\widehat{K}}\|F(\pi)\|_{op},\,F\in C_0(\widehat{K}).$
Its well-known that the $C^*$-algebra $C^*(K)$ of $K$ is isomorphic to $C_0(\widehat{K})$ (see, [1]).\\
In the sequel, we describe the elements of the image of the Fourier transform of $C^*(G_0)$ inside the big algebra
$\ell^\infty(\widehat{G_0}).$
\begin{definition}\label{def1}
Let $(\mu,H)\in \Gamma_0\sqcup\Gamma_1$ and let $\rho_\mu\in\widehat{K_H}.$ We define the representation $\pi_{\mu,0}$ of $G_0$ by $$\pi_{\mu,0}:=Ind_{K_H\ltimes\mathfrak{p}}^{G_0}(\rho_\mu\otimes1)$$
and let $\mathcal{H}_{\mu,0}$ be its Hilbert space. By the Frobenius reciprocity, we obtain $$\pi_{\mu,0}=\widehat{\bigoplus}_{\lambda\geq\mu}\tau_\lambda$$
where $\lambda\geq\mu$ means that $\tau_\lambda\in Ind_{K_H}^K(\rho_\mu).$
\end{definition}
\section{\bf{The $C^*$-algebra of the group $G_0$.}}\label{sec:3}
\subsection{\bf The Fourier transform.}
Now, let $f\in L^1(G_0), h\in K$ and $\Psi\in\mathcal{H}_{\mu,H},$
then we can give the expression of the operator $\pi_{(\mu,H)}(f)$ by the following equality\begin{eqnarray*}
                                \pi_{(\mu,H)}(f)\Psi(h) &=& \int_{G_0}f(k,X)\pi_{(\mu,H)}(k,X)\Psi(h)dkdX \\
     &=& \int_K\int_\mathfrak{p}f(k,X)\Psi(k^{-1}h)e^{i\langle Ad(h^{-1})X,H\rangle}dkdX \\
     &=&\int_K\widehat{f}^2(hk^{-1},Ad(h)H)\Psi(k)dk  \\
      &=&\int_{K/K_H}\int_{K_H}\widehat{f}^2(hs^{-1}k^{-1},Ad(h)H)\rho_\mu(s^{-1})(\Psi(k))dsdk.
     \end{eqnarray*} Then
             \begin{eqnarray}\label{expoppimuH}
      \pi_{(\mu,H)}(f)\Psi(h)&=&\int_{K/K_H}f_{\mu,H}(h,k)(\Psi(k))dk,
   \end{eqnarray}
here $\widehat f^{2}$ denotes the partial Fourier transform of $f$ in the second variable and where \begin{eqnarray*}
                                                              f_{\mu,H}:K\times K &\longrightarrow& \mathcal{B}(\mathcal{H}_{\rho_\mu}) \\
        (h,k) &\longmapsto& \int_{K_H}\widehat f^2 (hsk^{-1},Ad(h)H)\rho_\mu(s)ds.
                                                                                                 \end{eqnarray*}

\begin{definition}\label{def2}
 Let
 \begin{eqnarray*}
  C_{0,2}(G_0):=\left\{f\in L^1(G_0)|\ \widehat f^2\in C_0(K\times\p)\right\}.
 \end{eqnarray*}
This space $C_{0,2}(G_0)$ is dense in $L^1(G_0)$ and hence also in $C^*(G_0).$
\end{definition}

\begin{definition}\label{def3}
For each $f\in C^*(G_0),$ the Fourier transform $\mathcal{F}(f)$ of $f$ is an isometric homomorphism on $C^*(G_0)$ into $\ell^\infty(\widehat{G_0})$ which is given by
\begin{eqnarray*}
  \mathcal{F}(f)(\mu,H) &=& \pi_{(\mu,H)}(f)\in\mathcal{B}(\mathcal{H}_{\mu,H}),\,\, (\mu,H)\in \Gamma_0\sqcup\Gamma_1, \\
  \mathcal{F}(f)(\lambda) &=& \tau_\lambda(f)\in\mathcal{B}(\mathcal{H}_{\lambda}) \forall \lambda\in\Gamma_2.
\end{eqnarray*}
\end{definition}
In the following Proposition, we give a description of elements $\mathcal{F}(f)$ for each $f\in C^*(G_0).$
\begin{proposition}\label{prop4}
Let $f\in C^*(G_0).$ Then we have:
\begin{enumerate}
  \item For every $(\mu,H)\in\Gamma_0\sqcup\Gamma_1,$ $\F(\mu,H)$ is a compact operator on $\mathcal{H}_{\mu,H}.$

  \item The mapping \begin{eqnarray*}
                      \widehat{G_0}  &\longrightarrow&\mathcal{B}(\mathcal{H}_{\pi_\gamma}) \\
                      \gamma&\longmapsto& \mathcal{F}(f)(\gamma)
                    \end{eqnarray*}
                    are norm continuous on the difference sets $\Gamma_i,\ i=0,\ 1,\ 2$..
  \item $\underset{|\mu|\longrightarrow+\infty}{\lim}\|\mathcal{F}(f)(\mu,H)\|_{op}=0$ .
  \item $\underset{H\longrightarrow 0}{\lim}\|\pi_{(\mu,H)}(f)-\pi_{\mu,0}(f)\|_{op}=0.$
\end{enumerate}
\end{proposition}
\begin{proof}
 Let $f\in C_{0,2}(G_0)$
 \begin{enumerate}
  \item We have the following bound for the kernel functions $f_{\mu,H}$ where $(\mu,H)\in\Gamma_0\cup\Gamma_1$
  \begin{eqnarray}\label{compactoper}
  \nn \left|\left|f_{\mu,H}(h,k)\right|\right|_{\text{op}}&\leq&\left|\left|\int_{K_H}\widehat f^2(hsk^{-1},Ad(h)H)\rho_\mu(s)ds\right|\right|_{\text{op}}\\
  \nn&\leq&\int_{K_H}|\widehat f^2(hsk^{-1},Ad(h)H)|\left|\left|\rho_\mu(s)\right|\right|_{\text{op}}ds\\
  \nn&\leq&\int_{K_H}|\widehat f^2(hsk^{-1},Ad(h)H)|ds\\
  &\leq&\no{\widehat f^2}_{\infty}
  \end{eqnarray}
Since the dimension of the representation $\rho_\mu$ of $K_H$ is finite (denote by $d_\mu$ ), it follows that the norme $\noop{\cdot}$ and $\no{\cdot}_{H.S}$
on the space $\H_{\rho_\mu}$ are equivalent and
$$\no{f_{\mu,H}(h,k)}_{H.S}\leq\sqrt{d_\mu}\noop{f_{\mu,H}(h,k)},\ \ h,k\in K.$$
Then by $(\ref{compactoper})$ the operator $\pi_{(\mu,H)}(f)$ is Hilbert-Schmidt and its Hilbert-Schmidt norm is given by
\begin{eqnarray}\label{noHSpimuH}
 \nn\no{\pi_{(\mu,H)}(f)}_{H.S}^2&=&\int_{K/K_H}\int_{K/K_H}\no{f_{\mu,H}(h,k)}_{H.S}^2dhdk\\
 &\leq& d_\mu\no{\widehat f^2}_\infty^2
\end{eqnarray}
\item Let $(\mu,H_n)_{n\in\N}$ be a sequence converges  to $(\mu,H)$ in $\Gamma_0$. Using now $(\ref{expoppimuH})$, we obtain
\begin{eqnarray*}
 & &\no{\pi_{\mu,H_n}(f)-\pi_{(\mu,H)}(f)}_{H.S}^2\\
 &=&\int_{K/K_H}\int_{K/K_H}\no{f_{\mu,H_n}(h,k)-f_{\mu,H}(h,k)}_{H.S}^2dhdk\\
 &\leq&d_\mu\int_{K/K_H}\int_{K/K_H}\noop{f_{\mu,H_n}(h,k)-f_{\mu,H}(h,k)}^2dhdk\\
 &\leq&d_\mu\int_{K/K_H}\int_{K/K_H}\left(\int_K|\widehat f^2(hsk^{-1},Ad(h)H_n)-\widehat{f}^2(hsk^{-1},Ad(h)H)|^2ds\right)dhdk,
\end{eqnarray*}
For $f\in C_{0,2}(G_0)$ we have
\begin{eqnarray*}
 \underset{n\to\infty}{\lim}|\widehat f^2(hsk^{-1},Ad(h)H_n)-\widehat{f}^2(hsk^{-1},Ad(h)H)|=0.
\end{eqnarray*}
Hence
\begin{eqnarray*}
 \underset{n\to\infty}{\lim}\no{\pi_{(\mu,H_n)}(f)-\pi_{(\mu,H)}(f)}_{H.S}=0.
\end{eqnarray*}
Let now $(\mu,H_n)_{n\in\N}$ be a sequence converges  to $(\mu,H)$ in $\Gamma_1$. We have for  $\xi\in L^2(K/K_H,\mu)$
\begin{eqnarray*}
 \no{(\pi_{(\mu,H_n)}(f)-\pi_{(\mu,H)}(f))\xi}_2^2&=&\int_{K}\left|\left|\int_K\left(f_{\mu,H_n}(h,k)-f_{\mu,H}(h,k)\right)(\xi(k))dk\right|\right|_{\H_\mu}^2dk\\
 &\leq&\underset{k\in K}{\sup}\left(\int_K\noop{f_{\mu,H_n}(h,k)-f_{\mu,H}(h,k)}dh\right)^{\frac{1}{2}}\\
 & &\times\underset{h\in K}{\sup}\left(\int_K\noop{f_{\mu,H_n}(h,k)-f_{\mu,H}(h,k)}dk\right)^{\frac{1}{2}}\no{\xi}_2,
\end{eqnarray*}
since the function $(k,h,s,H)\longmapsto|\widehat f^2 (hsk^{-1},Ad(h)H_n)-\widehat f^2 (hsk^{-1},Ad(h)H)|$ converges uniformly to $0$
as $n$ tends to $\infty$, since
\begin{eqnarray*}
 & &\noop{f_{\mu,H_n}(h,k)-f_{\mu,H}(h,k)}\\
 &=&\noop{\int_{K_{H_n}}\widehat f^2 (hsk^{-1},Ad(h)H_n)\rho_\mu(s)ds-\int_{K_H}\widehat f^2 (hsk^{-1},Ad(h)H)\rho_\mu(s)ds}\\
 &\leq&\noop{\int_{K_{H_n}}\widehat f^2 (hsk^{-1},Ad(h)H_n)\rho_\mu(s)ds-\int_{K_{H_n}}\widehat f^2 (hsk^{-1},Ad(h)H)\rho_\mu(s)ds}\\
 &+&\noop{\int_{K_{H_n}}\widehat f^2 (hsk^{-1},Ad(h)H)\rho_\mu(s)ds-\int_{K_{H}}\widehat f^2 (hsk^{-1},Ad(h)H)\rho_\mu(s)ds}\\
 &\leq&\int_{K_{H}}|\widehat f^2 (hsk^{-1},Ad(h)H_n)-\widehat f^2 (hsk^{-1},Ad(h)H)|ds\\
 &+&\int_{K}|\widehat f^2 (hsk^{-1},Ad(h)H)\vert |{\bf1}_{K_{H_n}}(s)-{\bf1}_{K_H}(s)|ds
\end{eqnarray*}
we see therefore that
$$\underset{n\to\infty}{\lim}\noop{\pi_{(\mu,H_n)}(f)-\pi_{(\mu,H)}(f)}=0$$
uniformly in $\mu$.
\item It remains for us to prove that $\underset{|\mu|\to\infty}{\lim}\no{\pi_{(\mu,H)}(f)}_{H.S}=0.$ This is equivalent to show that
$$\underset{|\mu|\to\infty}{\lim}\no{f_{\mu,H}(k,h)}_{H.S}=0\ \text{ for all } k,h\in K.$$
Recall that
$$f_{\mu,H}(h,k)= \int_{K_H}\widehat f^2 (hsk^{-1},Ad(h)H)\rho_\mu(s)ds.$$
Let
\begin{eqnarray*}
 \varphi_H^{h,k}(s)=\widehat f^2 (hsk^{-1},Ad(h)H), \ \ s\in K_H.
\end{eqnarray*}
So,
$$f_{\mu,H}(h,k)=\rho_\mu(\varphi_H^{h,k}).$$
Using now the Plancherel formula, we get
\begin{eqnarray*}
 \no{\varphi_H^{h,k}}_{L^2(K_H)}^2=\underset{\mu\in\widehat{K_H}}{\sum}d_\mu\no{\rho_\mu(\varphi_H^{h,k})}_{H.S}^2.
\end{eqnarray*}
Therefore
$$\underset{|\mu|\to\infty}{\lim}\no{f_{\mu,H}(h,k)}_{H.S}=\underset{|\mu|\to\infty}{\lim}\no{\rho_\mu(\varphi_H^{h,k})}_{H.S}=0.$$
\item From $(\ref{expoppimuH})$ we have for $\mu\in\widehat{K_H}$ and $\xi\in L^2(K/K_H,\mu)$
\begin{eqnarray*}
 \no{(\pi_{(\mu,H)}(f)-\pi_{\mu,0}(f))\xi}_2^2&=&\int_{K}\left|\left|\int_K\left(f_{\mu,H}(h,k)-f_{\mu,0}(h,k)\right)(\xi(k))dk\right|\right|_{\H_\mu}^2dk\\
 &\leq&\underset{k\in K}{\sup}\left(\int_K\noop{f_{\mu,H}(h,k)-f_{\mu,0}(h,k)}dh\right)^{\frac{1}{2}}\\
 & &\times\underset{h\in K}{\sup}\left(\int_K\noop{f_{\mu,H}(h,k)-f_{\mu,0}(h,k)}dk\right)^{\frac{1}{2}}\no{\xi}_2,
\end{eqnarray*}
where
\begin{eqnarray*}
 \left.\begin{array}{cccc}
f_{\mu,0}: & K\times K& \longrightarrow& \B(H_\mu)\\
&(h,k)&\mapsto&\int_{K_H}\widehat f^2 (hsk^{-1},0)\rho_\mu(s)ds\end{array}\right.
\end{eqnarray*}
Since the function $(k,h,s,H)\longrightarrow|\widehat f^2 (hsk^{-1},Ad(h)H)-\widehat f^2 (hsk^{-1},0)|$ converges uniformly to $0$ as $H$ tends to $0$ since
\begin{eqnarray*}
 \noop{f_{\mu,H}(h,k)-f_{\mu,0}(h,k)}&=&\noop{\int_{K_H}(\widehat f^2 (hsk^{-1},Ad(h)H)-\widehat f^2 (hsk^{-1},0))\rho_\mu(s)ds}\\
 &\leq&\int_{K_H}|\widehat f^2 (hsk^{-1},Ad(h)H)-\widehat f^2 (hsk^{-1},0)|ds
\end{eqnarray*}
we see therefore that
$$\underset{H\to0}{\lim}\noop{\pi_{(\mu,H)}(f)-\pi_{\mu,0}(f)}=0$$
uniformly in $\mu$.
 \end{enumerate}
 The proposition follows now from the density of $C_{0,2}(G_0)$ in $C^*(G_0)$.
\end{proof}
\subsection{\bf{A $C^*$-condition.}}
\begin{definition}\label{def4}
 For an operator field $F\in\ell^\infty(\widehat{G_0}),$ let
 $$F(\mu,0):=\underset{\la\geq\mu}{\bigoplus}F(\la)\in\B\left(\underset{\la\geq\mu}{\bigoplus}\H_\mu\right).$$
 We have
 $$\noop{F(\mu,0)}=\underset{\la\geq\mu}{\sup}\noop{F(\la)}.$$
\end{definition}
\begin{definition}\label{def5}
 Let $\D$ be the familly consisting of all operator fields $F\in\ell^\infty(\widehat{G_0})$ satisfying the following conditions:
 \begin{enumerate}
  \item $F(\mu,H)$ is a compact operator on $\H_{\mu,H}$ for every $(\mu,H)\in\Gamma_0\sqcup\Gamma_1.$
  \item The mapping $\widehat{G_0}\longrightarrow\B(\H_{\mu,H}):(\mu,H)\longmapsto F(\mu,H)$
  are norm continuous on the difference sets $\Gamma_i,\ i=0,\ 1,\ 2$.
  \item $\underset{|\mu|\to\infty}{\lim}\noop{F(\mu,H)}=0.$
  \item $\underset{H\to0}{\lim}\noop{F(\mu,H)-F(\mu,0)}=0$ uniformly in $\mu\in K_H.$
  \item $\underset{|\la|\to\infty}{\lim}\noop{F(\la)}=0.$
 \end{enumerate}

\end{definition}
\begin{definition}\label{def6}
 Let
\begin{eqnarray*}
 \D_0=\{F\in\D;\ F(\la)=0 \text{for all }\la\in\widehat K\}.
\end{eqnarray*}

\end{definition}
\begin{remark}\label{rem1}
 According to Proposition $4.7$ in \cite{Lud-Ell-Abd} the unitary dual $\widehat{\D_0}$ is in bijection with the parameter space $\Gamma_0\cup\Gamma_1.$

\end{remark}

\begin{theorem}\label{the1}
 $\D$ is a $C^*$-algebra for the norm $\noop{\cdot},$ which is isomorphic to the $C^*$-algebra of Cartan motion groups under the Fourier transform.
\end{theorem}
\begin{proof}
 First we show that $\D$ is a $C^*$-algebra. It is clear that $\D$ is an involutive sub-algebra of $\ell^{\infty}(\widehat{G_0})$. The conditions
 $(1),\ (2),\ (3)$ et $(5)$ are evidently true for every $F\in\D$. For the condition $(4)$, let $(F_k)_k\subset\D$ such that
 $\underset{k\to\infty}{\lim}\no{F_k-F}_\infty=0.$ We have then $\underset{H\to 0}{\lim}\noop{F_k(\mu,H)-F_k(\mu,0)}=0$ uniformly in $\mu$, since
 for any $\varepsilon>0$ there exists $k_0$ such that such that $\no{F-F_k}\infty<\varepsilon$ for any $k\geq k_0.$ Then
 $\noop{F_{k_0}(\mu,H)-F_{k_0}(\mu,0)}<\varepsilon$ uniformly in $\mu$. Hence
 \begin{eqnarray*}
 & &\noop{F(\mu,H)-F(\mu,0)}\\
 &\leq&\noop{F(\mu,H)-F_{k_0}(\mu,H)}+\noop{F_{k_0}(\mu,H)-F_{k_0}(\mu,0)}+\noop{F_{k_0}(\mu,0)-F(\mu,0)}\\
 &\leq&3\varepsilon.
 \end{eqnarray*}
Then $F\in\D$.

Second, we have the quotient space $\D/\D_0$ is isomorphic to $C^*(K)$, indeed, let
\begin{eqnarray*}
 \left.\begin{array}{cccc}
 & & &\\
\de: & \D& \longrightarrow& C^*(K)\\
&F&\mapsto&\left(F(\la)\right)_{\la\in \widehat K}\end{array}\right.
\end{eqnarray*}
We have $\ker{\de}=\D_0.$ Since $\D$ contains the algebra $\F(C^*(G_0))$ the image of $\de$ contains
the image of the Fourier transform of $C^*(K)$. Then $\de$ is surjective. In addition $\D$ is a type $I$ algebra
and $\widehat{G_0}$ is in bijection with $\widehat\D$, indeed we have $\widehat{G_0}\subset\widehat\D$.
Moreover for every for every irreducible representation $\pi\in\widehat\D$, we have either $\pi\res{\D_0}\ne\{0\}$, and then $\pi=\pi_{(\mu,H)}$ since
$\D_0$ is closed ideal of $\D$ and $\widehat{\D_0}=\Gamma_0\bigcup\Gamma_1$ or $\pi\in\widehat{\D/\D_0}$ and $\pi\in\Gamma_2$. Hence $\pi\in\widehat{G_0}$
and $\widehat{G_0}=\widehat\D$ as sets.

Thanks to Stone-Weierstrass's theorem we have the identity $\D=\F(C^*(G_0)).$
\end{proof}

\end{document}